\documentclass{article}
\usepackage{amsmath,amsthm,mathrsfs,amssymb,enumerate}
\newtheorem{thm}{Theorem}[section]

\newtheorem{cor}[thm]{Corollary}

\numberwithin{equation}{section}
\begin{document}

\title{The index of coincidence for the binomial distribution is log-convex}

\author{I. Ra\c{s}a\footnote{Department of Mathematics, Technical University of Cluj-Napoca, Memorandumului Street 28, 400114 Cluj-Napoca, Romania. E-mail: Ioan.Rasa$@$math.utcluj.ro}}
\date{}
\maketitle

\textbf{Abstract}
We consider the binomial distribution with parameters $n$ and $x$,
and show that the sum of the squared probabilities is a log-convex
function of $x$. This completes the proof of a conjecture formulated in
2014. Applications to R\'{e}nyi and Tsallis entropies are given. \newline

\section{Introduction}
A family $\left ( p_{n,k}^{[c]}(x) \right )_{k=0,1,\dots}$ of probability distributions was considered in~\cite{3}; see also the references therein. The numbers $n>0$ and $c$ are real, subject to some conditions, and $x$ is a parameter in a certain interval $I_c$. The cases $c=-1,0,1$ correspond, respectively, to the binomial, Poisson, and negative binomial distributions.

It was conjectured in~\cite{4} that the index of coincidence
\begin{equation*}
S_{n,c}(x):= \sum _{k=0}^\infty \left ( p_{n,k}^{[c]}(x) \right )^2, \quad x \in I_c,
\end{equation*}
is logarithmically convex on $I_c$. This conjecture was validated in~\cite{1} for $c\geq 0$.

\section{Main results}
The aim of this note is to prove the conjecture for $c<0$. Without loss of generality we may restrict to the case $c=-1$. Briefly, we shall prove the following
\begin{thm}
The function
\begin{equation*}
F_n(x):=S_{n,-1}(x) = \sum _{k=0}^n \left ( {n \choose k}x^k (1-x)^{n-k} \right )^2
\end{equation*}
is log-convex on $[0,1]$.
\end{thm}
\begin{proof}
The proof is inspired by the method used in~\cite{2} in order to prove that $F_n$ is convex on $[0,1]$.

First, we need the inequalities
\begin{equation}
0 \leq u_n(t) \leq \frac{2n^2}{\sqrt{4n^2(t^2-1)+(t-\sqrt{t^2-1})^2}+t-\sqrt{t^2-1}}, \quad t\geq 1\label{eq:1},
\end{equation}
where $u_n:=P_n'/P_n$ and $P_n(t)$ are the classical Legendre polynomials. The first inequality is well-known (see, e.g. \cite[(1.2)]{2}). We prove the second one by induction with respect to $n$.

It is easy to verify it for $n=1$. Suppose that it is true for a certain $n\geq 1$. Then, according to~\cite{2}, (3.2) and the subsequent remark,
\begin{equation*}
u_{n+1}(t) = (n+1) \frac{n+1+tu_n(t)}{(n+1)t+(t^2-1)u_n(t)}\leq
\end{equation*}
\begin{equation*}
\leq (n+1) \frac{(n+1)\left [ \sqrt{4n^2(t^2-1)+(t-\sqrt{t^2-1})^2} + t - \sqrt{t^2-1} \right ] + 2n^2t}{(n+1)t \left [ \sqrt{4n^2(t^2-1)+(t-\sqrt{t^2-1})^2} + t - \sqrt{t^2-1} \right ] + 2n^2(t^2-1)} \leq
\end{equation*}
\begin{equation*}
\frac{2(n+1)^2}{\sqrt{4(n+1)^2(t^2-1)+(t-\sqrt{t^2-1})^2} + t - \sqrt{t^2-1}},
\end{equation*}
where the last inequality can be proved by a straightforward calculation using, e.g., the substitution $t = y / \sqrt{y^2-1}$, $y>1$. Thus~\eqref{eq:1} is completely proved.

Now let $x\in \left [0, \frac{1}{2}\right )$ and $t = \frac{2x^2-2x+1}{1-2x}\geq 1$. Then, according to \cite[(2.3) and (2.4)]{2},
\begin{equation}
\frac{F'_n(x)}{F_n(x)} = \frac{2\sqrt{t^2-1}}{t-\sqrt{t^2-1}}\left ( u_n(t)-\frac{n}{\sqrt{t^2-1}} \right ).\label{eq:2}
\end{equation}

Denote $X:=x(1-x)$, $X'=1-2x$, and let $z_1<z_2$ be the roots of the equation
\begin{equation*}
XX'z^2 + \left [ 1+4(n-1)X \right ]z+2nX'=0.
\end{equation*}

By using~\eqref{eq:1} and \eqref{eq:2} we obtain $z_1 \leq \frac{F_n'(x)}{F_n(x)}\leq z_2$, hence
\begin{equation}
XX' \frac{(F_n'(x))^2}{F_n(x)} + \left [ 1+4(n-1)X \right ]F'_n(x) + 2nX'F_n(x)\leq 0. \label{eq:3}
\end{equation}

On the other hand, it was proved in \cite[(33)]{3}, \cite[(4.13)]{4} that $F_n$ is a solution of the differential (Heun) equation
\begin{equation}
XX' F''_n(x) + \left [ 1+4(n-1)X \right ]F'_n(x) + 2nX'F_n(x) = 0. \label{eq:4}
\end{equation}

From \eqref{eq:3} and \eqref{eq:4} we get $(F'_n(x))^2 \leq F''_n(x)F_n(x)$, and so $F_n$ is log-convex on $\left [ 0, \frac{1}{2} \right ]$. To conclude the proof, it suffices to remark that $F_n(1-x) = F_n (x)$ for all $x\in [0,1]$.
\end{proof}

The R\'{e}nyi and Tsallis entropies of order 2 corresponding to the binomial distribution are defined, respectively, by $R_n(x) = -\log {F_n(x)}$ and $T_n(x) = 1-F_n(x)$, $x\in [0,1]$.

So we have the following
\begin{cor}
$R_n$ is concave and $T_n$ is log-concave on $[0,1]$.
\end{cor}

\end{document}